\documentclass[11pt]{article}
\usepackage{amsfonts}

\usepackage{amssymb,amsthm,amsmath,hyperref,txfonts}

\usepackage{graphicx,float}

\numberwithin{equation}{section}

\newtheorem{thm}{Theorem}[section]
\newtheorem{lem}{Lemma}[section]
\newtheorem{cor}{Corollary}[section]

\newtheorem{rem}{Remark}[section]

\theoremstyle{definition}

\newcommand{\D}{\displaystyle}

\newcommand{\beqno}{\begin{eqnarray*}}
\newcommand{\eeqno}{\end{eqnarray*}}

\allowdisplaybreaks

\begin{document}
\title{\bf Global Classical Solutions of  Viscous Liquid-gas Two-phase Flow Model }
\author{
Haibo Cui, \ Huanyao Wen\thanks{Corresponding author. E-mail:
huanyaowen@hotmail.com}, \ Haiyan Yin
\\[5mm]\
\textit{\small  Department of Mathematics, Central China Normal
University, Wuhan 430079, China}}

\date{}
\maketitle
\begin{abstract}
In this paper, we consider the global existence and uniqueness of
the classical solutions for the  3D  viscous liquid-gas two-phase
flow model. Initial data is only small in the energy-norm. Our main
ideas
 come from \cite{Z} where the existence of global classical
solutions to the compressible Navier-Stokes equations was obtained
by using the continuity methods under the assumption that the
initial energy is sufficiently small.

\vspace{4mm}


\end{abstract}

\section{Introduction}
In this paper, we consider the following
viscous liquid-gas two-phase flow model
    \begin{equation}\label{2dbu-E1.1}
    \left\{
    \begin{array}{l}
        m_t+\mathrm{div}(m u)=0,\\
        n_t+\mathrm{div}(n u)=0,\\
 (m u)_t+\mathrm{div}(m u\otimes u)+\nabla P(m, n)=
            \mu\Delta
            u+(\mu+\lambda)\nabla \mathrm{div}u, \ \ {\rm in} \ \  \mathbb{R}^3\times (0, \infty),
    \end{array}
    \right.
    \end{equation}
with  the  initial and boundary conditions

    \begin{equation}
      (m, n, u)|_{t=0}=(m_0, n_0, u_0)(x),  \ \ \mathrm{in} \ \  \mathbb{R}^3, \label{2dbu-E1.2}
    \end{equation}

     \begin{equation}\label{2dbu-E1.3}
     u(x, t)\rightarrow0,\ \  m(x,t)\rightarrow \tilde{m}>0,\ \  n(x,t)\rightarrow \tilde{n}>0,\ \ \mathrm{as}\ \  |x|\rightarrow\infty,\ \ \
     t\ge0.
    \end{equation}
 Here $m=\alpha_l\rho_l$ and  $n=\alpha_g\rho_g$ denote liquid mass
and gas mass,  respectively;  $\mu$, $\lambda$ are viscosity
constants, satisfying
\begin{equation}\label{2dbu-E1.4}
\mu>0,\ \   2\mu+3\lambda\geq 0,
\end{equation} which deduces $\mu+\lambda>0$.
  The unknown
 variables $\alpha_l$, $\alpha_g\in[0,1]$ denote liquid and gas volume
 fractions respectively, satisfying the fundamental relation:
$\alpha_l+ \alpha_g=1.$ Furthermore, the other unknown variables $\rho_l$
and $\rho_g$ denote liquid and gas
densities respectively, satisfying equations of state:
$\rho_l=\rho_{l, 0}+\frac{P-P_{l, 0}}{a_l^2}$,
$\rho_g=\frac{P}{a_g^2}$, where $a_l,$ $a_g$ are sonic speeds,
respectively, in liquid and gas, and $P_{l,0}$ and $\rho_{l, 0}$ are respectively
the reference pressure and density given as constants; $u$ denotes
velocities of liquid and gas; $P$ is common
   pressure for both phases, which satisfies
   \begin{equation}\label{2dbu-E1.5}
  P(m, n)=C^0\left(-b(m, n)+\sqrt{b(m, n)^2+c(m, n)}\right),
   \end{equation}
with $C^0=\frac{1}{2}a_l^2$, $k_0=\rho_{l, 0}-\frac{P_{l,
0}}{a_l^2}>0$, $a_0=(\frac{a_g}{a_l})^2$ and
\begin{equation*}
b(m, n)=k_0-m-\left(\frac{a_g}{a_l}\right)^2n=k_0-m-a_0n,
\end{equation*}
\begin{equation*}
c(m, n)=4k_0\left(\frac{a_g}{a_l}\right)^2n=4k_0a_0n.
\end{equation*}

The detailed  explanations about the above model can refer to
\cite{Y.Z.Z}, we omit it here.

We should mention that the methods introduced by Evje and Karlsen in
\cite{EV1}, Yao, Zhang and Zhu in \cite{Y.Z.Z} for the two-phase flow model
and Hoff in \cite{Hoff2005,Hoff1995}, Zhang and Fang in \cite{Z.F},
Zhang in \cite{Z} for the single-phase Navier-Stokes equations will
play crucial roles in our proof here.

As in \cite{EV1}, we give the potential energy function $G$ in the
form
\begin{equation}\label{2dbu-w1.1}
G(m,\frac{n}{m})=m\int_{\tilde{m}}^m\frac{P(s,\frac{n}{m}s)-P(\tilde{m},\tilde{n})}{s^2}
ds+\frac{m}{\tilde{m}}P(\tilde{m},\tilde{n})-\frac{m}{\tilde{m}}P(\tilde{m},\frac{n}{m}\tilde{m}).
\end{equation}
Now we assume that the initial data $(m_0,n_0,u_0)$ will be measured in the norm given by
\begin{equation}\label{2dbu-w1.2}
E_0=\int\left(\frac{1}{2}m_0|u_0|^2+G\left(m_0,\frac{n_0}{m_0}\right)\right)dx.
\end{equation}
Let
\begin{equation}\label{2dbu-w1.3}
M=\int|\nabla u_0|^2dx.
\end{equation}
It follows that there is a constant $q$, which will be fixed throughout, such that
\begin{equation}\label{2dbu-w1.4}
q\in (1,\frac{4}{3}),\ \ \
   with \ \ \
q^2<\frac{4\mu}{\mu+\lambda}, \ \ \   and\ \ \ \lambda<3\mu.
\end{equation}
 The vorticity matrix and the
effective viscous flux are defined respectively as follows:
\begin{equation}\label{2dbu-w1.9}
\omega^{j,k}=\partial_ku^j-\partial_ju^k,
\end{equation}
and
\begin{equation}\label{2dbu-w1.5}
F=(\lambda+2\mu)\mathrm{div}u-P(m,n)+P(\tilde{m},\tilde{n}).
\end{equation}
From (\ref{2dbu-w1.9}) and (\ref{2dbu-w1.5}), we have
\begin{equation}\label{2dbu-w1.55}
\Delta
u^j=\partial_j(\frac{F+P(m,n)-P(\tilde{m},\tilde{n})}{\lambda+2\mu})+\partial_i(\omega^{j,i}).
\end{equation}
Finally, we  denote the material derivative $\frac{D}{Dt}$ by
$\frac{D\omega}{Dt}=\dot{\omega}=\omega_t+u\cdot\nabla\omega$ for
function $\omega(x,t)$.

\noindent

The following is the main result of this paper.

\begin{thm}\label{2dbu-T1.1}
For sufficiently small constants $\varepsilon\in(0,1)$,
$\underline{m}_0>0$, $\overline{m}_0>0$, $\underline{n}_0$ and
$\overline{n}_0$, with
$\underline{m}_0\leq\tilde{m}\leq\overline{m}_0$,
$\underline{n}_0\leq\tilde{n}\leq\overline{n}_0$, let the initial
data $(m_0(x),n_0(x),u_0(x))$ satisfy
\begin{equation}\label{2dbu-w1.6}
\left\{
    \begin{array}{l}
 \underline{m}_0\leq\inf\limits_{x}m_0\leq\sup\limits_{x}m_0\leq\overline{m}_0, \\
 \underline{n}_0\leq\inf\limits_{x}n_0\leq\sup\limits_{x}n_0\leq\overline{n}_0, \\
 0<E_0\leq\varepsilon, \\
 m_0-\tilde{m}, n_0-\tilde{n}, u_0\in H^3.
      \end{array}
    \right.
\end{equation}
Define
\begin{equation}
\underline{s}_0=\inf\limits_{x}\frac{n_0}{m_0}, \ \
\overline{s}_0=\sup\limits_{x}\frac{n_0}{m_0}.
\end{equation}
Furthermore, assume that
\begin{equation}\label{2dbu-w1.7}
\overline{s}_0=\frac{\tilde{n}}{\tilde{m}}.
\end{equation}
Then there exist constants $\underline{m}$, $\overline{m}$, with
$\underline{m}<\underline{m}_0<\overline{m}_0<\overline{m}$ and
$\overline{m}>\frac{\tilde{n}}{\underline{s}_0}$, such that the
problem (1.1)-(1.3) has a unique global classical solution
$(m,n,u)(x,t)$ satisfying
\begin{equation}\label{2dbu-w1.6}
\left\{
    \begin{array}{l}
     0<\underline{m}\leq m(x,t)\leq\overline{m}, \\
     \underline{s}_0\underline{m}\leq n(x,t)\leq\overline{s}_0\overline{m},\\
   (m-\tilde{m}, n-\tilde{n}, u)\in C^1(\mathbb{R}^3\times(0,T])\cap C([0,T],H^3)\cap C^1((0,T],H^2),
      \end{array}
    \right.
\end{equation}
furthermore, we have
\begin{equation}\label{2dbu-w1.7}
\sup\limits_{t\in[0,T]}\|(m-\tilde{m},n-\tilde{n},u)\|_{H^3}+\sup\limits_{t\in[0,T]}\|(m_t,n_t)\|_{H^2}+\int_0^T\|u\|_{H^4}^2dt\leq
C(T),\ \ \ \forall  T>0.
\end{equation}
\begin{equation}\label{2dbu-w1.8}
\sup\limits_{t\in(\tau,T]}\|u_t\|_{H^2}\leq C(\tau, T),\ \ \ \forall \tau>0, \  T>0.
\end{equation}
\end{thm}

\begin{rem}\label{2dbu-L1.2}
It is easy to verify
 \begin{equation}
\left\{
\begin{array}{l}\label{2dbu-E1.13}
P_m=\frac{\partial P}{\partial m}=C^0\left\{1-\frac{b}{\sqrt{b^2+c}}\right\}>0, \\
 P_n=\frac{\partial P}{\partial n}=C^0\left\{a_0+\frac{a_0}{\sqrt{b^2+c}}(m+a_0n+k_0)\right\}>0,
\ \ m, \ n
>0.\end{array}\right.
 \end{equation}
 This shows that  $P(m, n)$ is increasing in $m$ and $n$ for $m, \ n>0.$
 \end{rem}
 \begin{rem}\label{remark 1.2}
It should be mentioned that the existence of global strong solutions
for 3D with vacuum was obtained recently by Guo, Yang and Yao,
please see \cite{Guo-Yang-Yao}, where the initial energy was assumed
to be small enough and the solutions satisfied \beqno
&0\le\underline{s}_0m\le n\le\overline{s}_0m,\ (m-\tilde{m},
n-\tilde{n})\in C([0,T]; W^{1,q_0}\cap H^1),&\\& u\in
C([0,T];D_0^1\cap D^2)\cap L^2(0,T;D^{2,q_0}),\ u_t\in
L^2(0,T;D_0^1),\ \sqrt{m}u_t\in L^\infty(0,T; L^2),& \eeqno for some
$q_0\in(3,6]$. It seems impossible to consider the existence of
classical solutions under the assumptions of \cite{Guo-Yang-Yao},
since higher order derivatives of the pressure function are
unbounded on $\{(m,n)|m=k_0, n=0\}$, such as
\begin{equation}\label{Guo-yang-yao solution}
\partial_n^2
P(m,n)=\frac{-4C^0a_0^2k_0m}{\Big[(k_0-m-a_0n)^2+4k_0a_0n\Big]^\frac{3}{2}}=\infty,\
\mathrm{on}\ \{(m,n)|m=k_0, n=0\}.
\end{equation} It seems that the assumption $\inf n_0>0$ and $\inf
m_0\ge0$ is enough. While, for simplicity, we assume that both $n_0$
and $m_0$ are positive in Theorem \ref{2dbu-T1.1}. In this case, the
compatibility condition like (1.16) in \cite{Guo-Yang-Yao} is not
necessary.

 \end{rem}

\section{The proof of Theorem \ref{2dbu-T1.1}}\label{2dbu-S6}
The local existence of the solutions to
(\ref{2dbu-E1.1})-(\ref{2dbu-E1.3}) with the regularities as in
Theorem \ref{2dbu-T1.1} can be obtained by the similar methods as in
\cite{C}, \cite{W-Zhu2} and the references therein. We omit it here
for brevity. The regularities guarantee the uniqueness (refer for
instance to \cite{Hoff1995}). Let $[0, T^*)$ be the maximal
existence interval of the above solutions. Note that the local
existence of the solutions guarantees $T^*>0$. Our goal is to prove
$T^*=\infty$ by using a contradiction argument. More precisely,
suppose $T^*<\infty$, our aim is to get
\begin{equation}\label{aim1}
\begin{array}{l}
\sup\limits_{t\in[0,T]}\|(m-\tilde{m}, n-\tilde{n}, u)\|_{H^3}\le K,
\end{array}
\end{equation} and
\begin{equation}\label{aim2}
\begin{array}{l}
\inf\limits_{(x,t)\in\mathbb{R}^3\times[0,T]}m\ge \frac{1}{K}, \inf\limits_{(x,t)\in\mathbb{R}^3\times[0,T]}n\ge \frac{1}{K},
\end{array}
\end{equation}
for any $T\in(0,T^*)$, where $K$ is a generic positive constant depending only on $T^*$ and other known constants but independent of $T$. With (\ref{aim1}) and (\ref{aim2}),
 $T^*$ is not the maximal existence time of the solutions, which is the desired contradiction.

The proof is divided into two steps. It should be pointed that the
$H^1\times H^2$-estimates of $((m,n), u)$ could be obtained by the
same arguments as in \cite{Guo-Yang-Yao}. For completeness, we still
present some of the crucial estimates which might be slightly
different from those as in \cite{Guo-Yang-Yao} with $m$ and $n$
positive lower bounds. For the higher order estimates of $(m,n,u)$,
we shall apply some ideas which were used to handle the 3D
single-phase Navier-Stokes equations, see for instance \cite{Z}.
More precisely,
we proceed as follows.\\

 {\noindent\bf Step 1:} The bounds of the density.

 {\bf Claim:} There exist $\varepsilon\in(0,1)$ sufficiently small and $\underline{m}\in(0,\underline{m}_0)$, and $\overline{m}\in(\overline{m}_0,\infty)$, and
$\overline{m}>\frac{\tilde{n}}{\underline{s}_0}$, such that for any given $T\in(0,T^*)$, the following estimates hold:
\begin{equation}\label{2dbu-c1.1}
\underline{m}\leq m(x,t)\leq \overline{m},\ (x,t)\in \mathbb{R}^3\times[0,T],
\end{equation}
\begin{equation}\label{2dbu-c1.2}
A_1(T)+A_2(T)\leq2E_0^\theta,
\end{equation}
for some $\theta\in(0,1)$, provided the initial energy $E_0\le\varepsilon$. Here we have used the following two notations:
\begin{equation}\label{2dbu-c1.3}
 \begin{array}[b]{l}
A_1(T)=\D\sup\limits_{t\in[0,T]}\sigma\int|\nabla u|^2dx+\int_0^T\int\sigma|\dot{u}|^2dxdt,\vspace{.2cm} \\
A_2(T)=\D\sup\limits_{t\in[0,T]}\sigma^{{3}}\int|\dot{u}|^2dx+\int_0^T\int\sigma^{{3}}|\nabla\dot{u}|^2dxdt,
\end{array}
\end{equation} where $\sigma=\sigma(t)=\min\{1,t\}$.

Define
$$
T_0=\sup\left\{S\in[0,T]\Big| \ \underline{m}\leq m(x,t)\leq \overline{m},\ (x,t)\in \mathbb{R}^3\times[0,S],\ \mathrm{and}\ A_1(S)+A_2(S)\leq2E_0^\theta\right\}.
$$
To get (\ref{2dbu-c1.1}) and (\ref{2dbu-c1.2}), it suffices to prove $T_0=T$.

Since $\underline{m}<m_0(x)<\overline{m}$ and $A_1(0)=A_2(0)=0$, we
get $T_0>0$ by using the continuity of $m$, $A_1(t)$ and $A_2(t)$
with respect to $t$ over $[0,T]$.

To get $T_0=T$, it suffices to prove
\begin{equation*}
\underline{m}<\underline{m}_1\leq m(x,t)\leq \overline{m}_1<
\overline{m},\ \mathrm{for}\ \mathrm{some}\ \mathrm{constants}\ \underline{m}_1\ \mathrm{and}\  \overline{m}_1,
\end{equation*}
and
\begin{equation*}
A_1(t)+A_2(t)\leq E_0^\theta,\ \ \ \forall t\in[0,T_0],
\end{equation*} provided the initial energy $E_0\le\varepsilon$, for $\varepsilon>0$ sufficiently small.\\

From the momentum equation (\ref{2dbu-E1.1})$_3$ and
(\ref{2dbu-w1.9}), (\ref{2dbu-w1.5}), we have
\begin{equation}\label{2dbu-c1.4}
m\dot{u}^j=\partial_jF+\mu\partial_k\omega^{j,k},
\end{equation}
which implies
\begin{equation}\label{2dbu-c1.5}
\Delta F=\mathrm{div}(m\dot{u}),
\end{equation}
and
\begin{equation}\label{2dbu-c1.6}
\mu\Delta
\omega^{j,k}=\partial_k(m\dot{u}^{j})-\partial_j(m\dot{u}^{k}).
\end{equation}

This shows that the $L^2$ estimate of $m\dot{u}$ implies $L^2$ bounds of $\nabla F$ and $\nabla\omega$.
 Equations (\ref{2dbu-c1.5}) and (\ref{2dbu-c1.6}) will play important roles in this section.

Throughout the rest of the paper, we denote the generic constant by $C$ depending on the initial data and other known constants,
but independent of $T_0$, $T$ and $T^*$. We omit the integration domain when we integrate
some functions over $\mathbb{R}^3$.

\begin{lem}\label{2dbu-a1.0}
Under the conditions of Theorem 1.1, it holds that
\begin{equation}\label{m,n}\frac{m}{C}\leq n \leq C  m.\end{equation}
\end{lem}
\begin{proof}
The proof of Lemma \ref{2dbu-a1.0} can be found by Yao-Zhang-Zhu in
\cite{Y.Z.Z}.

\end{proof}

\begin{lem}\label{2dbu-a1.1}
Under the conditions of Theorem 1.1, it holds that
\begin{equation}\label{2dbu-q1.1}
\sup\limits_{t\in[0,T_0]}\int\left(|u|^2+(m-\tilde{m})^2+(n-\tilde{n})^2\right)dx+\int_0^{T_0}\int|\nabla
u|^2dxdt\leq CE_0.
\end{equation}
\end{lem}
\begin{proof}
Let
\begin{equation}
A(t)=\int\left\{\frac{1}{2}m|u|^2+G\left(m,\frac{n}{m}\right)\right\}dx.
\end{equation}
Differentiating $A(t)$ with respect to $t$, using integration by
parts and the equation $(1.1)$, we get (\ref{2dbu-q1.1}).

\end{proof}

\begin{lem}\label{2dbu-a1.2}
Under the conditions of Theorem 1.1, it holds that
\begin{equation}\label{2dbu-b1.2}
A_1(T_0)\leq CE_0+C\int_0^{T_0}\int\sigma|\nabla u|^3dxdt,
\end{equation}
and
\begin{equation}
\begin{array}{rl}\label{2dbu-b1.3}
A_2(T_0)\leq\D
CE_0+C\int_0^{T_0}\int\sigma|\nabla u|^3dxdt+C\int_0^{T_0}\int\sigma^{{3}}|\nabla u|^4dxdt.
\end{array}
 \end{equation}
\end{lem}
\begin{proof}
The estimates (\ref{2dbu-b1.2}) and (\ref{2dbu-b1.3}) can been
obtained by the similar arguments as that  in \cite{Y.Z.Z}.

\end{proof}

 To handle the higher order terms on the right-hand sides of (\ref{2dbu-b1.2}) and (\ref{2dbu-b1.3}), we need the following lemma whose proof can be found in \cite{Y.Z.Z} and references therein.

\begin{lem}\label{2dbu-a1.3}
Under the conditions of Theorem 1.1, it holds that
\begin{equation}\label{2dbu-b1.4}
\|u\|_{L^p}\leq C_p\|u\|_{L^2}^{\frac{6-p}{2p}}\|\nabla
u\|_{L^2}^{\frac{3p-6}{2p}},\ \ \ \ p\in[2,6].
\end{equation}
\begin{equation}\label{2dbu-b1.5}
\|u\|_{L^p}^p\leq C_pE_0^\frac{6-p}{4}\|\nabla
u\|_{L^2}^{\frac{3p-6}{2}},\ \ \ p\in[2,6].
\end{equation}
\begin{equation}\label{2dbu-b1.6}
\|\nabla u\|_{L^r}\leq C_r(\| F\|_{L^r}+\|\omega\|_{L^r}+\|
P(m,n)-P(\tilde{m},\tilde{n})\|_{L^r}),\ \ r\in(1,\infty).
\end{equation}
\begin{equation}\label{2dbu-b1.7}
\|\nabla F\|_{L^r}+\|\nabla\omega\|_{L^r}\leq C_r\|m
\dot{u}\|_{L^r},\ \ r\in(1,\infty).
\end{equation}
\begin{equation}\label{2dbu-b1.8}
\| F\|_{L^p}+\|\omega\|_{L^p}\leq C_p\|m
\dot{u}\|_{L^2}^{\frac{3p-6}{2p}}\left(\|\nabla
u\|_{L^2}^{\frac{6-p}{2p}}+\|
P(m,n)-P(\tilde{m},\tilde{n})\|_{L^2}^{\frac{6-p}{2p}}\right),\ \ \
p\in[2,6].
\end{equation}
Also, for $0\leq t_1\leq t_2 \leq T_0$, $l\geq2$ and  $s\geq0$, we
have
\begin{equation}\label{2dbu-b1.9}
\int_{t_1}^{t_2}\int
\sigma^s|P(m,n)-P(\tilde{m},\tilde{n})|^ldxds\leq
C\left(\int_{t_1}^{t_2}\int \sigma^s|F|^ldxds+E_0\right).
\end{equation}
\end{lem}

\begin{lem}\label{2dbu-a1.4}
Under the conditions of Theorem 1.1, there exists a constant $T_1>0$, such that
\begin{equation}\label{2dbu-b1.99}
\sup\limits_{t\in[0,T_0]}\int|\nabla u|^2dx+\int_0^{{T_0}\wedge
T_1}\int|\dot{u}|^2dxdt\leq C(1+M),
\end{equation} where we have used the notation $e_1\wedge
e_2=\min\{e_1, e_2\}$.
\end{lem}
\begin{proof}
Similar to the proof of (\ref{2dbu-b1.2}), multiplying $(1.1)_3$ by
$\dot{u}$, integrating the resulting equation over
$\mathbb{R}^3\times[0,t]$ ($t\in[0,T_0]$), and using integration by
parts, H\"older inequality, and  Cauchy inequality, we have
\begin{equation*}
\int|\nabla u|^2dx+\int_0^t\int|\dot{u}|^2dxds\leq
C(E_0+M)+C\int_0^t\int|\nabla u|^3dxds,
\end{equation*}
It follows from (\ref{2dbu-b1.6}) and (\ref{2dbu-b1.9}) that
\begin{equation*}
\int_0^t\int|\nabla u|^3dxds\leq
C+C\int_0^t\int(|F|^3+|\omega|^3)dxds.
\end{equation*}
By (\ref{2dbu-b1.8}), we get
\begin{equation}
\begin{array}{rl}\label{2dbu-b1.93}
\D\int(|F|^3+|\omega|^3)dx\leq C\D\left(\int(|\nabla
u|^2+|P(m,n)-P(\tilde{m},\tilde{n})|^2)dx\right)^{\frac{3}{4}}\left(\int
m|\dot{u}|^2dx\right)^{\frac{3}{4}}.
\end{array}
\end{equation}
Thus, from Lemma \ref{2dbu-a1.1} and Young inequality with
$\varepsilon$, we have
\begin{equation*}
\begin{array}{rl}
&\D\int|\nabla u|^2dx+\int_0^t\int|\dot{u}|^2dxds\vspace{.2cm}\\
\leq &\D C(1+M)+C\D\int_0^t\left(\int(|\nabla u|^2+|P(m,n)-P(\tilde{m},\tilde{n})|^2)dx\right)^{\frac{3}{4}}\left(\int m|\dot{u}|^2dx\right)^{\frac{3}{4}}ds\vspace{.2cm} \\
\leq & C(1+M)+C\D\int_0^t\left(\int(|\nabla u|^2+|P(m,n)-P(\tilde{m},\tilde{n})|^2)dx\right)^3ds+\frac{1}{2}\int_0^t\int |\dot{u}|^2dxds\vspace{.2cm} \\
\leq & C\D(1+M)+Ct\sup\limits_{s\in[0,t]}\|\nabla u(\cdot,s)\|_{L^2}^6+\frac{1}{2}\int_0^t\int |\dot{u}|^2dxds\vspace{.2cm} \\
\leq & C\D(1+M)+Ct(1+M)^2\|\nabla
u(\cdot,s)\|_{L^2}^{2}+\frac{1}{2}\int_0^t\int |\dot{u}|^2dxds,
\end{array}
\end{equation*} for $t\in(0,1)\cap(0,T_0)$ sufficiently small, where we have used the continuity of $\int|\nabla u|^2(x,t)\,dx$ with respect to $t$ over $[0,T_0]$.
Taking $T_1=\min\{\frac{1}{8C(1+M)^{2}},1\}$ and letting $t\le T_0\wedge
T_1$, we obtain (\ref{2dbu-b1.99}).
\end{proof}

\begin{lem}\label{2dbu-a1.5}
Under the conditions of Theorem 1.1, for $\varepsilon>0$ sufficiently small, we have
\begin{equation}
 \begin{array}{l}\label{2dbu-b1.10}
\D\sup\limits_{t\in[0,T_0]}\int(\sigma|\nabla
u|^2+\sigma^3|\dot{u}|^2)dx+\int_0^{T_0}\int(\sigma|\dot{u}|^2+\sigma^3|\nabla\dot{u}|^2)dxdt\leq
E_0^\theta.
\end{array}
\end{equation}
\end{lem}
\begin{proof}
From (\ref{2dbu-b1.2}) and (\ref{2dbu-b1.3}), we have
\begin{equation}\label{2dbu-b1.111}
LHS \ of \ (\ref{2dbu-b1.10})\leq CE_0+C\int_0^{T_0}\int(\sigma|\nabla
u|^3+\sigma^3|\nabla u|^4)dxds.
\end{equation}
By (\ref{2dbu-b1.6}), we have
\begin{equation}\label{2dbu-b1.1}
\int_0^{T_0}\int\sigma^3|\nabla u|^4dxds\leq
C\int_0^{T_0}\int\sigma^3[|F|^4+|\omega|^4+|P(m,n)-P(\tilde{m},\tilde{n})|^4]dxds.
\end{equation}
Using (\ref{2dbu-c1.2}), (\ref{2dbu-q1.1}), (\ref{2dbu-b1.4}) and
(\ref{2dbu-b1.7})-(\ref{2dbu-b1.9}), we have
\begin{equation}
\begin{array}[b]{rl}\label{2dbu-b1.11}
&\D\int_0^{T_0}\int\sigma^3\left(|F|^4+|\omega|^4\right)dxds\vspace{.2cm}\\
\leq & \D C
\int_0^{T_0}\sigma^3\left[\left(\int|F|^2dx\right)^{\frac{1}{2}}\left(\int|\nabla
F|^2dx\right)^{\frac{3}{2}}
+\left(\int|\omega|^2dx\right)^{\frac{1}{2}}\left(\int|\nabla\omega|^2dx\right)^{\frac{3}{2}}\right]ds\vspace{.2cm} \\
\leq & C \D\int_0^{T_0}\sigma^3\left(\|\nabla u\|_{L^2}+\|\nabla
P(m,n)-P(\tilde{m},\tilde{n})\|_{L^2}\right)\|m  \dot{u}\|_{L^2}^3ds\vspace{.2cm} \\
\leq & C\sup\limits_{t\in[0,T_0]}\left\{ \int\sigma(|\nabla u|^2+|P(m,n)-P(\tilde{m},\tilde{n})|^2)dx\int\sigma^3m |\dot{u}|^2ds\right\}^{\frac{1}{2}}\vspace{.2cm} \\
&\ \ \ \D \times\int_0^{T_0}\int\sigma m |\dot{u}|^2dxds\vspace{.2cm} \\
\leq & CE_0^{2\theta}.
\end{array}
\end{equation}
From (\ref{2dbu-b1.11}), we have
\begin{equation}
 \begin{array}[b]{ll}\label{2dbu-b1.12}
\D\int_0^{T_0}\int\sigma^3|P(m,n)-P(\tilde{m},\tilde{n})|^4dxds&\leq
C\left(\D\int_0^{T_0}\int\sigma^3|F|^4dxds+E_0\right)
\vspace{.2cm}\\
&\leq  CE_0^{2\theta}+CE_0.
\end{array}
\end{equation}
From (\ref{2dbu-b1.1})-(\ref{2dbu-b1.12}), we have
\begin{equation}\label{2dbu-b1.13}
\int_0^{T_0}\int\sigma^3|\nabla u|^4\leq  CE_0^{2\theta}+CE_0.
\end{equation}
From (\ref{2dbu-b1.13}), (\ref{2dbu-q1.1}), we get
\begin{equation}
 \begin{array}[b]{ll}\label{2dbu-b1.14}
\D\int_{{T_0}\wedge T_1}^{{T_0}}\int\sigma|\nabla u|^3dxds& \leq
\frac{1}{2}\int_{{T_0}\wedge T_1}^{{T_0}}\int(\sigma^2|\nabla
u|^4+|\nabla u|^2)dxds \vspace{.2cm} \\& \leq
C(T_1)\int_{{T_0}\wedge T_1}^{{T_0}}\int(\sigma^3|\nabla
u|^4+|\nabla u|^2)dxds \vspace{.2cm} \\ &\leq
C(M)E_0^{2\theta}+C(M)E_0,
\end{array}
\end{equation}
where we have used Lemma \ref{2dbu-a1.4}.

From (\ref{2dbu-c1.2}), (\ref{2dbu-q1.1}), (\ref{2dbu-b1.6}),
(\ref{2dbu-b1.8}), (\ref{2dbu-b1.9}) and Lemma \ref{2dbu-a1.4}, we
have
\begin{equation}
\begin{array}[b]{rl}\label{2dbu-b1.15}
&\D\int_0^{{T_0}\wedge T_1}\int\sigma|\nabla u|^3dxds\vspace{.2cm} \\
\leq & \D CE_0+\int_0^{{T_0}\wedge
T_1}\int\sigma(|F|^3+|\omega|^3)dxds
\vspace{.2cm} \\
\leq & CE_0+\int_0^{{T_0}\wedge T_1}\sigma\left(\int(|\nabla
u|^2+|P(m,n)-P(\tilde{m},\tilde{n})|^2)dx\right)^{\frac{3}{4}}\left(\int
m|\dot{u}|^2dx\right)^{\frac{3}{4}}ds \vspace{.2cm} \\ \leq
&C(M)E_0+ C\left\{\sup\limits_{t\in[0,T_0]}\sigma\|\nabla
u\|_{L^2}^4\int_0^{{T_0}\wedge T_1}\|\nabla
u\|_{L^2}^2ds+\int_0^{{T_0}\wedge
T_1}\sigma\|P(m,n)-P(\tilde{m},\tilde{n})\|_{L^2}^6dt\right\}^{\frac{1}{4}}\vspace{.2cm} \\
&\ \ \ \D\times \left(\int_0^{{T_0}\wedge
T_1}\sigma\|m\dot{u}\|_{L^2}^2ds \right)^{\frac{3}{4}}
\vspace{.2cm} \\
\leq &
C(M)E_0+C(M)E_0^{\frac{1}{4}+\theta}+C(M)E_0^{\frac{3}{4}(1+\theta)}.
\end{array}
\end{equation}
Then, from (\ref{2dbu-c1.2}), (\ref{2dbu-b1.111}), (\ref{2dbu-b1.13})-(\ref{2dbu-b1.15}), we obtain
\begin{equation}
LHS \ of \ (\ref{2dbu-b1.10})\leq
C(M)E_0^{1\wedge2\theta\wedge\frac{3}{4}(1+\theta)\wedge(\frac{1}{4}+\theta)}.
\end{equation}
Thus, when $\varepsilon$ is sufficiently small such that $
C(M)\varepsilon^{(1-\theta)\wedge\theta\wedge(\frac{3}{4}-\frac{1}{4}\theta)\wedge\frac{1}{4}}\leq1
$, we can get
\begin{equation}
LHS \ of \ (\ref{2dbu-b1.10})\leq E_0^{\theta}.
\end{equation}

This completes the proof of Lemma \ref{2dbu-a1.5}.
\end{proof}

From Lemma \ref{2dbu-a1.4},  Lemma \ref{2dbu-a1.5} and  Ref.
\cite{Z} (Propositions 3-5), we can get the next lemma.

\begin{lem}\label{2dbu-a1.6}
Under the conditions of Theorem 1.1, it holds
\begin{equation}\label{2dbu-b1.16}
\sup\limits_{t\in[0,T_0]}\int|\nabla
u|^2dx+\int_0^{T_0}\int|\dot{u}|^2dxdt\leq C(M).
\end{equation}
 If we assume further
that there exists $q\in(1,\frac{4}{3})
 $ satisfying $q^2<\frac{4\mu}{\lambda+\mu}$,
then we have
\begin{equation}\label{2dbu-b1.18}
\sup\limits_{t\in[0,T_0]}\sigma^{p_1}\int|\dot{u}|^{2+q}dx+\int_0^{T_0}\int\sigma^{p_1}|\dot{u}|^q|\nabla\dot{u}|^2dxdt\leq
C(M),\ \ \ p_1= 1+\frac{5q}{4}
  .
\end{equation}
\end{lem}

\begin{lem}\label{2dbu-a1.7}
Under the conditions of Theorem 1.1, it holds
\begin{equation}\label{2dbu-b1.19}
\|F\|_{L^\infty}+\|\omega\|_{L^\infty}\leq C(\|\nabla
u\|_{L^2}+\|m-\tilde{m}\|_{L^2}
+\|n-\tilde{n}\|_{L^2})^{\frac{2q-2}{4+5q}}\|m
\dot{u}\|_{L^{2+q}}^{\frac{6+3q}{4+5q}},
\end{equation}
and
\begin{equation}\label{2dbu-b1.20}
\int_0^{T_0}(\|F\|_{L^\infty}+\|\omega\|_{L^\infty})ds\leq
C(M)E_0^{\frac{\theta(q-1)}{4+5q}}(1+T_0), \ \ \ q\in (1,\frac{4}{3}).
\end{equation}
\begin{proof}
From (\ref{2dbu-q1.1}), (\ref{2dbu-b1.7}), (\ref{2dbu-b1.10}), and
the Gagliardo-Nirenberg inequality, we have
\begin{equation}
\begin{array}[b]{ll}\label{2dbu-b1.21}
\|F\|_{L^\infty} & \leq C\|F\|_{L^2}^{\frac{2(q-1)}{4+5q}}\|\nabla
F\|_{L^{2+q}}^{\frac{6+3q}{4+5q}} \vspace{.2cm} \\
&\leq C(\|\nabla
u\|_{L^{2}}+\|m-\tilde{m}\|_{L^{2}}+\|n-\tilde{n}\|_{L^{2}})^{\frac{2(q-1)}{4+5q}}\|m
\dot{u}\|_{L^{2+q}}^{\frac{6+3q}{4+5q}},
\end{array}
\end{equation}
and
\begin{equation}
\begin{array}[b]{ll}\label{2dbu-b1.22}
\D\int_0^{T_0}\|F\|_{L^\infty}ds&\leq\int_0^{T_0}(\|\nabla
u\|_{L^{2}}+\|m-\tilde{m}\|_{L^{2}}+\|n-\tilde{n}\|_{L^{2}})^{\frac{2(q-1)}{4+5q}}\|m
\dot{u}\|_{L^{2+q}}^{\frac{6+3q}{4+5q}}ds \vspace{.2cm}
\\&\leq
C(M)\int_0^{T_0}(\sigma^{-\frac{1}{2}}E_0^{\frac{\theta}{2}})^{\frac{2(q-1)}{4+5q}}(\sigma^{-\frac{p_1}{2+q}})^{\frac{6+3q}{4+5q}}ds
\vspace{.2cm} \\&\leq C(M)E_0^{\frac{\theta(q-1)}{4+5q}}(1+{T_0}).
\end{array}
\end{equation}
Similarly, we can obtain the same estimates for $\omega$. This
completes the proof of Lemma \ref{2dbu-a1.7}.
\end{proof}
\end{lem}
Now, we apply the estimates in Lemmas
\ref{2dbu-a1.7}-\ref{2dbu-a1.8} and the hypothesis (\ref{2dbu-c1.1})
to close the bounds of $m$.
\begin{lem}\label{2dbu-a1.8}

Under the conditions of Theorem 1.1, for given constants $\underline{m}_1$ and $\overline{m}_1$ satisfying
$0<\underline{m}<\underline{m}_1<\tilde{m}<\overline{m}_1<\overline{m}$ and $\overline{m}_1\geq\frac{\tilde{n}}{\underline{s}_0}$, there exists a constant $\varepsilon>0$ sufficiently small, such that
\begin{equation}\label{2dbu-b1.23}
\underline{m}_1\leq m(x,t)\leq \overline{m}_1,\ \ \  (x,t)\in
\mathbb{R}^3\times[0,T_0],
\end{equation}
provided that $E_0\leq \varepsilon$.
 Furthermore, the estimates in Lemma \ref{2dbu-a1.1}-\ref{2dbu-a1.7} hold.
\begin{proof}
Using the similar argument as that in Ref. \cite{Y.Z.Z} (Proposition
2.5) and Ref. \cite{Z} (Proposition 7), we can easily obtain this
lemma in $\mathbb{R}^3$ and omit the details.
\end{proof}
\end{lem}
By (\ref{2dbu-b1.10}) and (\ref{2dbu-b1.23}), we get $T_0=T$. Thus, (\ref{2dbu-c1.1}) and (\ref{2dbu-c1.2}) hold for any $T\in(0,T^*)$. This ends the proof of  {\bf Step 1}.\\

{\noindent\bf Step 2:} Estimates for the higher order derivatives of ($m$, $n$, $u$).\\

\noindent Just as   in \cite{Sun-Wang-Zhang, Sun-Zhang}, we
introduce the quantity $w$, which is defined by
$$
w=u-v,
$$
where $v$ is the solution of
 \begin{equation}\label{2dbu-E3.10}
    \left\{
    \begin{aligned}
      & \mu \Delta v+(\lambda+\mu)\nabla \mathrm{div} v=\nabla P(m, n) \ \ \mathrm{in} \ {\mathbb{R}^3},\vspace{.2cm} \\
     & v(x)=0  \ \ \ \ \mathrm{as}\ \  |x|\rightarrow\infty.
    \end{aligned}
    \right.
    \end{equation}
The following estimates can be found in the ref. \cite{Sun-Wang-Zhang} (Proposition 2.1):
 \begin{eqnarray}\label{2dbu-E3.11}
  \left\{
    \begin{aligned}
      &  \|\nabla v\|_{L^p}\leq C\|P(m, n)-P(\tilde{m}, \tilde{n})\|_{L^p},\vspace{.2cm} \\
& \|\nabla^2v\|_{L^p}\leq C\|\nabla P(m, n)\|_{L^p},
\end{aligned}
    \right.
 \end{eqnarray} for any $p\in(1, \infty)$.

  By using the equations (\ref{2dbu-E1.1}), we find $w$ satisfies
 \begin{equation}\label{2dbu-E3.12}
    \left\{
    \begin{aligned}
      & \mu \Delta w+(\lambda+\mu)\nabla \mathrm{div} w=m \dot{u} \ \ \mathrm{in} \ {\mathbb{R}^3},\vspace{.2cm} \\
     & w(x)=0  \ \ \ \ \mathrm{as}\ \  |x|\rightarrow\infty.
    \end{aligned}
    \right.
    \end{equation}

\begin{lem}\label{2dbu-a1.9}
Under the conditions of Theorem 1.1, it holds that
\begin{equation}\label{2dbu-b1.80}
\sup\limits_{t\in[0,T]}\int
m|\dot{u}|^2dx+\int_0^T\int\left(|\nabla\dot{u}|^2+|\frac{D}{Dt}\mathrm{div}{u}|^2\right)dxdt
\leq K.
\end{equation}
\end{lem}
 \begin{proof}
We take the operator $\partial_t+\mathrm{div}(u\cdot)$ in $(1.1)_3$,
multiplying the resulting equations by $\dot{u}$, we have
\begin{equation}
\begin{array}[b]{cl}
&\ \ \
{\dot{u}}^j[\partial_t(m\dot{u}^j)+\mathrm{div}(um\dot{u}^j)]+\dot{u}^j[\partial_j
P_t+\mathrm{div}(u\partial_j P)]\vspace{.2cm}
\\&=\mu{\dot{u}}^j[\partial_t\Delta
{u}^j+\mathrm{div}(u\Delta u^j)]
+(\lambda+\mu)\dot{u}^j[\partial_j\partial_t(\mathrm{div}u)+\mathrm{div}(u\partial_j(\mathrm{div}u))].
\end{array}
\end{equation}
Integrating the above equation over $\mathbb{R}^3$ and using
integration by parts, we have
\begin{equation}\label{2dbu-b1.78}
\begin{array}[b]{rl}
\D\frac{1}{2}\int m|\dot{u}|^2dx=&\D\frac{1}{2}\int
m_0|\dot{u}_0|^2dx-\int_0^t\int\dot{u}^j[\partial_jP_t+\mathrm{div}(u\cdot\partial_jP)]dxds
+\int_0^t\int\mu\dot{u}^j[\Delta u_t^j\vspace{.2cm} \\
&\D+\mathrm{div}(u\cdot\Delta u^j)]dxds+\int_0^t\int(\lambda+\mu)\dot{u}^j[\partial_t\partial_j\mathrm{div}u+\mathrm{div}(u\cdot\partial_j\mathrm{div}u)]dxds\vspace{.2cm} \\
:=&\sum\limits_{i=1}^{4}K_i.
\end{array}
\end{equation}
From $m_0-\tilde{m}\in H^1$ and $u_0\in H^2$, we know
\begin{equation}\label{{2dbu-b1.81}}
K_1=\frac{1}{2}\int m_0|\dot{u}_0|^2dx\leq K.
\end{equation}
From the integration by part, the equation $(1.1)_1$, $(1.1)_2$,
(\ref{2dbu-q1.1}), (\ref{2dbu-b1.23}) and the Cauchy inequality, we
get
\begin{equation}\label{2dbu-b1.77}
\begin{array}[b]{ll}
K_2&\D=-\int_0^t\int\dot{u}^j[\partial_jP_t+\mathrm{div}(u\partial_jP)]dxds\vspace{.2cm} \\
   &\D=\int_0^t\int\left(\partial_j\dot{u}^j(P_mm_t+P_nn_t)+\partial_k\dot{u}^j\partial_jPu^k\right)dxds\vspace{.2cm} \\
   &\D=-\int_0^t\int \left(P_m(m\mathrm{div}u+u\cdot\nabla m)\partial_j\dot{u}^j +P_n(n\mathrm{div}u+u\cdot\nabla n)\partial_j\dot{u}^j\right)dxds\vspace{.2cm} \\
     &\ \ \ -\D\int_0^t\int P(m,n)\partial_j(\partial_k\dot{u}^ju^k)dxds\vspace{.2cm} \\
   &\D=\int_0^t\int\left(-P_mm\mathrm{div}u\partial_j\dot{u}^j-P_nn\mathrm{div}u\partial_j\dot{u}^j+\partial_k(\partial_j\dot{u}^ju^k)P\vspace{.2cm} -\partial_j(\partial_k\dot{u}^ju^k)P\right)dxds\vspace{.2cm} \\
   &\D=\int_0^t\int\left(-P_mm\mathrm{div}u\partial_j\dot{u}^j-P_nn\mathrm{div}u\partial_j\dot{u}^j
   +\partial_j\dot{u}^j\mathrm{div}uP-\partial_k\dot{u}^j\partial_j{u}^kP\right)dxds\vspace{.2cm} \\
    &\D\leq KE_0+\frac{\mu}{4} \int_0^t\int|\nabla \dot{u}|^2dxds.
\end{array}
\end{equation}
From the integration by part and the Cauchy inequality, we get
\begin{equation}\label{2dbu-b1.76}
\begin{array}[b]{ll}
K_3&\D=\mu\int_0^t\int\dot{u}^j[\Delta{u}_t^j+\mathrm{div}(u\Delta u^j)]dxds\vspace{.2cm} \\
   &\D=-\mu\int_0^t\int[\partial_i\dot{u}^j\partial_i{u}_t^j+\Delta u^j u\cdot\nabla \dot{u}^j]dxds\vspace{.2cm} \\
   &\D=-\mu\int_0^t\int \left(|\nabla \dot{u}|^2-\partial_i\dot{u}^ju^k\partial_k\partial_iu^j-\partial_i\dot{u}^j\partial_iu^k\partial_ku^j+\Delta u^j u\cdot\nabla \dot{u}^j\right)dxds\vspace{.2cm} \\
   &\D=-\mu\int_0^t\int\left( |\nabla \dot{u}|^2+\partial_i\dot{u}^j\partial_ku^k\partial_iu^j-\partial_i\dot{u}^j\partial_iu^k\partial_ku^j-\partial_i u^j \partial_iu^k\partial_k \dot{u}^j\right)dxds\vspace{.2cm} \\
   &\D\leq -\frac{\mu}{2} \int_0^t\int|\nabla \dot{u}|^2dxds+K\int_0^t\int |\nabla
   u|^4dxds.
\end{array}
\end{equation}
From the integration by part, (\ref{2dbu-q1.1}), (\ref{2dbu-b1.23})
and the Cauchy inequality, we get
\begin{equation}\label{2dbu-b1.75}
\begin{array}[b]{ll}
K_4&\D=(\lambda+\mu)\int_0^t\int \dot{u}^j[\partial_t\partial_j\mathrm{div}u+\mathrm{div}(u\partial_j\mathrm{div}u)]dxds\vspace{.2cm} \\
   &\D=-(\lambda+\mu)\int_0^t\int\left(\partial_j\dot{u}^j[\partial_t(\mathrm{div}u)+\mathrm{div}(u\mathrm{div}u)]
   +\dot{u}^j\mathrm{div}(\partial_ju\mathrm{div}u)\right)dxds\vspace{.2cm} \\
   &\D=-(\lambda+\mu)\int_0^t\int\partial_j\dot{u}^j\frac{D}{Dt}\mathrm{div}udxds-(\lambda+\mu)\int_0^t\int\partial_j\dot{u}^j(\mathrm{div}u)^2dxds\vspace{.2cm} \\
   &\ \ \ \ \D
   -(\lambda+\mu)\int_0^t\int\dot{u}^j\partial_i(\partial_ju^i\mathrm{div}u)dxds\vspace{.2cm} \\
   &\D=-(\lambda+\mu)\int_0^t\int|\frac{D}{Dt}\mathrm{div}{u}|^2dxds
   -(\lambda+\mu)\int_0^t\int\partial_j{u}^i\partial_i{u}^j\frac{D}{Dt}\mathrm{div}udxds\vspace{.2cm} \\
   &\ \ \ \ \D-(\lambda+\mu)\int_0^t\int\partial_j\dot{u}^j(\mathrm{div}u)^2dxds
   +(\lambda+\mu)\int_0^t\int\partial_i\dot{u}^j\partial_j{u}^i\mathrm{div}udxds\vspace{.2cm} \\
   &\D\le-\frac{(\lambda+\mu)}{2}\int_0^t\int|\frac{D}{Dt}\mathrm{div}{u}|^2dxds+\frac{\mu}{8}\int_0^t\int|\nabla\dot{u}|^2dxds+K\int_0^t\int|\nabla
   u|^4dxds.
\end{array}
\end{equation}
Using a similar argument as Lemma \ref{2dbu-a1.5}, we can get
\begin{equation}\label{2dbu-b1.37}
\int_0^t\int|\nabla
   u|^4dxds\leq K.
\end{equation}
From (\ref{2dbu-b1.78})-(\ref{2dbu-b1.37}) and the Cauchy
inequality, we can get
\begin{equation}\label{2dbu-b1.74}
\frac{1}{2}\int m|\dot{u}|^2dx+\frac{\mu}{8} \int_0^t\int|\nabla
\dot{u}|^2dxds+\frac{\lambda+\mu}{2}\int_0^t\int|\frac{D}{Dt}\mathrm{div}{u}|^2dxds
\leq K.
\end{equation}
We complete the proof of Lemma  \ref{2dbu-a1.9}.
\end{proof}

\begin{cor}\label{2dbu-C3.3}
 Under the conditions of Theorem 1.1, it holds that
 \begin{equation}\label{2dbu-t1.79}
\int_0^T\|\nabla w\|_{W^{1,l_1}}^2ds \leq K,\ \ \ where   \ \ \
l_1\in(3,6],\ \
 \ or\ \  l_1=2.
\end{equation}
 \end{cor}
 \begin{proof}
From (\ref{2dbu-b1.23}), (\ref{2dbu-E3.12}), (\ref{2dbu-b1.80}) and
Sobolev's embedding theorem, we have
\begin{equation}\label{2dbu-t1.1}
\begin{array}[b]{ll}
\D\int_0^T\|\nabla w\|_{W^{1,l_1}}^2ds&\leq\D K\int_0^T\| m \dot{u}\|_{L^{l_1}}^2ds\vspace{.2cm} \\
&\D\leq K\int_0^T\|\dot{u}\|_{L^{l_1}}^2ds\vspace{.2cm} \\
&\D\leq K\int_0^T\|\dot{u}\|_{H^1}^2ds\vspace{.2cm} \\
&\D\leq K,
\end{array}
\end{equation}
where we have used the standard elliptic estimate.
\end{proof}

\begin{lem}\label{2dbu-a1.10}
Under the conditions of Theorem 1.1, it holds that
\begin{equation}\label{2dbu-q1.98}
\sup\limits_{t\in[0,T]}(\|\nabla m(\cdot,t)\|_{L^{q_1}}+\|\nabla
n(\cdot,t)\|_{L^{q_1}})\leq K,\ \ \ \ \ q_1\in(3,6].\ \ \ \
\end{equation}
\end{lem}
 \begin{proof}
Differentiating the equation (\ref{2dbu-E1.1})$_1$ with respect to
$x_i$, then multiplying   both sides of the resulting equation by
$q_1|\partial_i m|^{q_1-2}\partial_im$, we get
\begin{eqnarray*}
&&\partial_t|\partial_i m|^{q_1} +\mathrm{div}(|\partial_i m|^{q_1} u)+(q_1-1)|\partial_i m|^{q_1}\mathrm{div} u\nonumber\vspace{.2cm} \\
&&+ {q_1} m|\partial_i m|^{{q_1}-2}\partial_im\partial_i\mathrm{div}
u+{q_1}|\partial_i m|^{{q_1}-2}\partial_im\partial_iu\cdot\nabla
m=0.
\end{eqnarray*}
Integrating the above equality over $\mathbb{R}^3$, we obtain
\begin{eqnarray}\label{2dbu-e4.3}
&&\frac{d}{dt}\int|\nabla m|^{q_1} dx\nonumber\vspace{.2cm} \\
&\leq &  K\int |\nabla u||\nabla m|^{q_1}dx+{q_1}\int m|\nabla \mathrm{div}u||\nabla m|^{{q_1}-1}dx\nonumber\vspace{.2cm} \\
&\leq & K\|\nabla u\|_{L^\infty}\|\nabla
m\|_{L^{q_1}}^{q_1}+K\|\nabla^2 u\|_{L^{q_1}}\|\nabla
m\|_{L^{q_1}}^{{q_1}-1}.
\end{eqnarray}
Similarly, we get
\begin{eqnarray}\label{2dbu-e4.4}
&&\frac{d}{dt}\int|\nabla n|^{q_1} dx\nonumber\vspace{.2cm} \\
&\leq & K\int |\nabla u||\nabla n|^{q_1}dx+{q_1}\int n|\nabla \mathrm{div}u||\nabla n|^{{q_1}-1}dx\nonumber\vspace{.2cm} \\
&\leq & K\|\nabla u\|_{L^\infty}\|\nabla
n\|_{L^{q_1}}^{q_1}+K\|\nabla^2 u\|_{L^{q_1}}\|\nabla
n\|_{L^{q_1}}^{{q_1}-1}.
\end{eqnarray}
From (\ref{2dbu-E3.11}), we obtain
\begin{equation}\label{2dbu-e4.5}
\|\nabla^2 v\|_{L^{q_1}}\leq K(\|\nabla m\|_{L^{q_1}}+\|\nabla
n\|_{L^{q_1}}),
\end{equation}
then we get
\begin{eqnarray}\label{2dbu-e4.6}
 \|\nabla v\|_{L^\infty}
&\leq &
K\Big(1+\|\nabla v\|_{BMO}\ln(e+\|\nabla^2v\|_{L^{q_1}})\Big)\nonumber\vspace{.2cm} \\
&\leq &K\Big(1+\|P\|_{L^{\infty}\cap L^2}\ln(e+\|\nabla P\|_{L^{q_1}})\Big)\nonumber\vspace{.2cm} \\
&\leq &K\Big(1+\ln(e+\|\nabla m\|_{L^{q_1}}+\|\nabla
n\|_{L^{q_1}})\Big),
\end{eqnarray}
where the first inequality could be found in \cite{Sun-Wang-Zhang}.

From (\ref{2dbu-e4.4})-(\ref{2dbu-e4.6}), we get
\begin{eqnarray*}\begin{split}
&\frac{d}{dt}\left(\|\nabla m\|_{L^{q_1}}+\|\nabla n\|_{L^{q_1}}\right)\vspace{.2cm} \\
\leq & K(1+\|\nabla w\|_{L^\infty}+\|\nabla v\|_{L^\infty})(\|\nabla
m\|_{L^{q_1}}
+\|\nabla n\|_{L^{q_1}})+C\|\nabla^2 w\|_{L^{q_1}}\vspace{.2cm} \\
\leq & K\left(1+\|\nabla w\|_{W^{1, {q_1}}}+\ln(e+\|\nabla
m\|_{L^{q_1}}
+\|\nabla n\|_{L^{q_1}})\right)(\|\nabla m\|_{L^{q_1}}+\|\nabla n\|_{L^{q_1}})\vspace{.2cm} \\
&+K\|\nabla^2 w\|_{L^{q_1}}.\end{split}
\end{eqnarray*}
Note that   $\|\nabla w\|_{W^{1, {q_1}}}\in L^2 (0,T)$ by  Corollary
\ref{2dbu-C3.3}. Then by the Gronwall's inequality, we obtain
(\ref{2dbu-q1.98}). This completes the proof of Lemma
\ref{2dbu-a1.10}.
\end{proof}

\begin{cor}\label{2dbu-C3.4}
Under the conditions of Theorem 1.1, it holds that
 \begin{equation}\label{2dbu-t1.78}
\int_0^T\|\nabla u\|_{L^{\infty}}^2ds \leq K.\ \ \ \
\end{equation}
 \end{cor}
 \begin{proof}
From (\ref{2dbu-t1.79}), (\ref{2dbu-q1.98}), (\ref{2dbu-e4.6}) and Sobolev's embedding theorem,  we have
\begin{equation}\label{{2dbu-t1.2}}
\begin{array}[b]{ll}
\D\int_0^T\|\nabla u\|_{L^{\infty}}^2ds&\leq\D \int_0^T\|\nabla w \|_{L^\infty}^2+\|\nabla v \|_{L^\infty}^2ds\vspace{.2cm} \\
&\D\leq K\int_0^T\|\nabla w \|_{W^{1,{q_1}}}^2ds+K\vspace{.2cm} \\
&\D\leq K.
\end{array}
\end{equation}
This completes the proof of Corollary \ref{2dbu-C3.4}.
\end{proof}

\begin{lem}\label{2dbu-a1.11}
Under the conditions of Theorem 1.1, it holds that
\begin{equation}\label{2dbu-q1.99}
\sup\limits_{t\in[0,T]}(\|\nabla m(\cdot,t)\|_{L^{2}}+\|\nabla
n(\cdot,t)\|_{L^{2}})\leq K.\ \ \ \ \
\end{equation}
\end{lem}
 \begin{proof}
Differentiating the equation (\ref{2dbu-E1.1})$_1$ with respect to
$x_i$, then multiplying  both sides of the resulting equation by
$2\partial_im$, we get
\begin{eqnarray*}
&&\partial_t|\partial_i m|^2 +\mathrm{div}(|\partial_i m|^2 u)+|\partial_i m|^2\mathrm{div} u\nonumber\vspace{.2cm} \\
&&+ 2 m\partial_im\partial_i\mathrm{div} u+2\partial_im\partial_iu\cdot\nabla m=0.
\end{eqnarray*}
Integrating the above equality over ${\mathbb{R}^3}$, we obtain
\begin{eqnarray}\label{2dbu-E4.3}
&&\frac{d}{dt}\int|\nabla m|^2 dx\nonumber\vspace{.2cm} \\
&\leq & K\int |\nabla u||\nabla m|^2dx+2\int m|\nabla \mathrm{div}u||\nabla m|dx\nonumber\vspace{.2cm} \\
&\leq & K\|\nabla u\|_{L^\infty}\|\nabla m\|_{L^2}^2+K\|\nabla^2
u\|_{L^2}\|\nabla m\|_{L^2}.
\end{eqnarray}
Similarly,
\begin{eqnarray}\label{2dbu-E4.4}
&&\frac{d}{dt}\int|\nabla n|^2 dx\nonumber\vspace{.2cm} \\
&\leq & K\int|\nabla u||\nabla n|^2dx+2\int n|\nabla \mathrm{div}u||\nabla n|dx\nonumber\vspace{.2cm} \\
&\leq & K\|\nabla u\|_{L^\infty}\|\nabla n\|_{L^2}^2+K\|\nabla^2
u\|_{L^2}\|\nabla n\|_{L^2}.
\end{eqnarray}
From (\ref{2dbu-E3.11}), we obtain
\begin{equation}\label{2dbu-E4.5}
\|\nabla^2 v\|_{L^2}\leq K(\|\nabla m\|_{L^2}+\|\nabla n\|_{L^2}).
\end{equation}
From (\ref{2dbu-e4.6}) and (\ref{2dbu-q1.98}), we get
\begin{equation}\label{2dbu-E4.6}
 \|\nabla v\|_{L^\infty}\leq K.
\end{equation}
From (\ref{2dbu-E4.3})-(\ref{2dbu-E4.6}), we get
\begin{eqnarray*}
&&\frac{d}{dt}\left(\|\nabla m\|_{L^2}+\|\nabla n\|_{L^2}\right)\nonumber\vspace{.2cm} \\
&\leq & K(1+\|\nabla w\|_{L^\infty})(\|\nabla m\|_{L^2}+\|\nabla
n\|_{L^2})
+K\|\nabla^2 w\|_{L^2}\nonumber\vspace{.2cm} \\
&\leq & K\left(1+\|\nabla w\|_{W^{1, q_1}}\right)(\|\nabla m\|_{L^2}+\|\nabla n\|_{L^2})\nonumber\vspace{.2cm} \\
&&+K\|\nabla^2 w\|_{L^2}.
\end{eqnarray*}
Note that   $\|\nabla w\|_{W^{1, 2}}$, $\|\nabla w\|_{W^{1, q_1}}\in
L^2 (0,T)$ by  Corollary \ref{2dbu-C3.3}. Then by the Gronwall's
inequality, we obtain (\ref{2dbu-q1.99}). This completes the proof
of Lemma \ref{2dbu-a1.11}.
\end{proof}

\begin{cor}\label{2dbu-C3.2}
 Under the conditions of Theorem 1.1, it holds that
 \begin{equation}\label{2dbu-b1.79}
\sup\limits_{t\in[0,T]}(\| u\|_{L^\infty}+\| u\|_{H^2}) \leq K.
\end{equation}
 \end{cor}
 \begin{proof}
From (\ref{2dbu-b1.7}) and (\ref{2dbu-b1.80}), we have
\begin{equation}
\begin{array}[b]{ll}\label{2dbu-b1.40}
\|\nabla F\|_{L^2}+\|\nabla \omega\|_{L^2}
&\leq K\|m \dot{u}\|_{L^2}\vspace{.2cm} \\
&\leq K.
\end{array}
\end{equation}
From (\ref{2dbu-w1.55}), (\ref{2dbu-q1.1}), (\ref{2dbu-q1.99}) and
(\ref{2dbu-b1.40}), we have
\begin{equation}\label{2dbu-b1.73}
\begin{array}[b]{ll}
\|u\|_{H^2}&\leq K(\|u\|_{L^2}+\|\nabla F\|_{L^2}+\|\nabla P\|_{L^2}+\|\nabla \omega\|_{L^2})\vspace{.2cm} \\
&\leq K(\|u\|_{L^2}+\|\nabla F\|_{L^2}+\|\nabla m\|_{L^2}+\|\nabla n\|_{L^2}+\|\nabla \omega\|_{L^2})\vspace{.2cm} \\
&\leq K.
\end{array}
\end{equation}
Then, from Sobolev's embedding theorem, we finish this proof of
Corollary \ref{2dbu-C3.2}.
 \end{proof}

\begin{lem}\label{2dbu-a1.12}
Under the conditions of Theorem 1.1, it holds that
\begin{equation}\label{2dbu-b1.19}
\sup\limits_{t\in[0,T]}\| m-\tilde{m}\|_{H^2}+\|
n-\tilde{n}\|_{H^2}\leq K.
\end{equation}
\begin{proof}
From (\ref{2dbu-E1.1})$_1$ and (\ref{2dbu-w1.5}), we have
\begin{equation}\label{2dbu-b1.20}
\partial_t\Lambda_1(m(x,t))+u\cdot\nabla\Lambda_1+P\left(m(x,t),n(x,t)\right)-P(\tilde{m},\tilde{n})=-F(x,t).
\end{equation}
where $\Lambda_1$ satisfies that $\Lambda_1(\tilde{m})=0$ and
$\Lambda_1'(m)=\frac{2\mu+\lambda}{m}>0$. Similarly, from
(\ref{2dbu-E1.1})$_2$ and (\ref{2dbu-w1.5}), we have
\begin{equation}\label{2dbu-b1.222}
\partial_t\Lambda_2(n(x,t))+u\cdot\nabla\Lambda_2+P\left(m(x,t),n(x,t)\right)-P(\tilde{m},\tilde{n})=-F(x,t).
\end{equation}
where $\Lambda_2$ satisfies that $\Lambda_2(\tilde{n})=0$ and
$\Lambda_2'(n)=\frac{2\mu+\lambda}{n}>0$.

 Differentiating
(\ref{2dbu-b1.20}) with respect to $x_i$ and $x_j$,
 multiplying both sides of the resulting equation by
$\partial_i\partial_j\Lambda_1(m)$, integrating the result equality
over $\mathbb{R}^3$, we obtain
\begin{equation}
\frac{1}{2}\frac{d}{dt}\int
|\partial_i\partial_j\Lambda_1(m)|^2dx\leq\int
\left(|\partial_i\partial_j\Lambda_1(m)||\partial_i\partial_jF|
+|\partial_i\partial_j\Lambda_1(m)||\partial_i\partial_jP|+|\partial_i\partial_j\Lambda_1(m)||\partial_i\partial_j(u\cdot\nabla\Lambda_1)|\right)dx.
\end{equation}
Using the Cauchy inequality and the Gagliardo-Nirenberg inequality,
we have
\begin{equation}\label{2dbu-b1.21}
\|\Lambda_1(t)\|_{H^2}^2\leq\|\Lambda_1(0)\|_{H^2}^2+K\int_0^t\left((\|F\|_{H^2}+\|\Lambda_2\|_{H^2})\|\Lambda_1\|_{H^2}+(
1+\|\nabla
u\|_{L^\infty}+\|u\|_{H^2})\|\Lambda_1\|_{H^2}^2\right)ds.
\end{equation}
Similar to (\ref{2dbu-b1.21}),  from  (\ref{2dbu-b1.222}), we have
\begin{equation}\label{2dbu-b1.211}
\|\Lambda_2(t)\|_{H^2}^2\leq\|\Lambda_2(0)\|_{H^2}^2+K\int_0^t\left((\|F\|_{H^2}+\|\Lambda_1\|_{H^2})\|\Lambda_2\|_{H^2}+(
1+\|\nabla
u\|_{L^\infty}+\|u\|_{H^2})\|\Lambda_2\|_{H^2}^2\right)ds.
\end{equation}
By (\ref{2dbu-c1.5}), (\ref{2dbu-q1.1}), (\ref{2dbu-b1.16}),
(\ref{2dbu-b1.23}), (\ref{2dbu-q1.99}), H\"older inequality and  the
Gagliardo-Nirenberg inequality,
 we get
\begin{equation}\label{2dbu-b1.223}
\begin{array}[b]{ll}
\|F\|_{H^2}&\leq K\left(\|F\|_{L^2}+\|\nabla m \dot{u}\|_{L^2}+\| m \nabla\dot{u}\|_{L^2}\right)\vspace{.2cm} \\
&\leq K\left(1+\|\nabla m \|_{L^3}\|\dot{u}\|_{L^6}+\| \nabla\dot{u}\|_{L^2}\right)\vspace{.2cm} \\
&\leq K\left(1+\|\nabla m \|_{L^2}^{\frac{1}{2}}\|\nabla^2 m
 \|_{L^2}^{\frac{1}{2}}\|\nabla\dot{u}\|_{L^2}+\| \nabla\dot{u}\|_{L^2}\right)\vspace{.2cm} \\
&\leq
K\left(1+\|\Lambda_1\|_{H^2}(1+\|\nabla\dot{u}\|_{L^2})+\|\nabla\dot{u}\|_{L^2}\right)
.
\end{array}
\end{equation}
Thus, from (\ref{2dbu-b1.80}), (\ref{2dbu-t1.78}),
(\ref{2dbu-b1.79}),
 (\ref{2dbu-b1.21}), (\ref{2dbu-b1.211}),
and (\ref{2dbu-b1.223}),  we have
\begin{equation}\label{2dbu-b1.224}
\begin{array}[b]{ll}
&\ \ \ \|\Lambda_1(t)\|_{H^2}^2+\|\Lambda_2(t)\|_{H^2}^2
\vspace{.2cm}\\
&\D\leq
\|\Lambda_1(0)\|_{H^2}^2+\|\Lambda_2(0)\|_{H^2}^2+K\int_0^t((1+\|\nabla
u\|_{L^\infty}+\|\nabla\dot{u}\|_{L^2}+\|u\|_{H^2})(\|\Lambda_1\|_{H^2}^2+\|\Lambda_2\|_{H^2}^2)+
\|\nabla\dot{u}\|_{L^2}^2)ds\vspace{.2cm} \\
&\leq K+K\D\int_0^t \left(1+\|\nabla
u\|_{L^\infty}+\|\nabla\dot{u}\|_{L^2}\right)\left(\|\Lambda_1\|_{H^2}^2+\|\Lambda_2\|_{H^2}^2\right)ds.
\end{array}
\end{equation}
From (\ref{2dbu-b1.23}), (\ref{2dbu-q1.99}), and using the
Gronwall's inequality, we can immediately obtain (\ref{2dbu-b1.19}).
\end{proof}
\end{lem}

\begin{lem}\label{2dbu-a1.13}
Under the conditions of Theorem 1.1, it holds that
\begin{equation}\label{2dbu-b6.7}
\int_0^T\| u\|_{H^3}^2dt\leq K.
\end{equation}
\end{lem}
\begin{proof}
From (\ref{2dbu-c1.6}), (\ref{2dbu-b1.80}), (\ref{2dbu-b1.19}) and
(\ref{2dbu-b1.223}), we have
\begin{equation}\label{2dbu-b1.24}
\int_0^T(\| F\|_{H^2}^2+\| \omega\|_{H^2}^2)dt\leq K.
\end{equation}
From (\ref{2dbu-w1.55}), (\ref{2dbu-q1.1}), (\ref{2dbu-b1.19}) and (\ref{2dbu-b1.24}), we have
\begin{equation}\label{2dbu-b1.25}
\begin{array}[b]{ll}
\D\int_0^T\| u\|_{H^3}^2dt&\leq K\D\int_0^T\left(\| u\|_{L^2}+\|\nabla F\|_{H^1}+\| \nabla (P(m,n)-P(\tilde{m},\tilde{n}))\|_{H^1}+\|\nabla \omega\|_{H^1}\right)^2dt\vspace{.2cm} \\
&\leq K\D\int_0^T\left(\| u\|_{L^2}+\| F\|_{H^2}+\| m-\tilde{m}\|_{H^2}+\| n-\tilde{n}\|_{H^2}+\| \omega\|_{H^2}\right)^2dt\vspace{.2cm} \\
&\leq K.
\end{array}
\end{equation}

This completes the proof of Lemma \ref{2dbu-a1.13}.
\end{proof}

\begin{lem}\label{2dbu-a1.14}
Under the conditions of Theorem 1.1, it holds that
\begin{equation}\label{2dbu-b1.26}
\sup\limits_{t\in[0,T]}\int| \nabla\dot{u}|^2dx+\int_0^T\int|
\nabla^2\dot{u}|^2dxdt\leq K.
\end{equation}
\end{lem}
\begin{proof}
We take the operator $\nabla\partial_t+\nabla\mathrm{div}(u\cdot)$
in $(1.1)_3$, multiplying the resulting equations by $\nabla \dot{u}$,  we obtain
\begin{equation}
\begin{array}[b]{cl}
&\ \ \
\nabla{\dot{u}}^j\nabla[\partial_t(m\dot{u}^j)+\mathrm{div}(um\dot{u}^j)]+
\nabla\dot{u}^j\nabla[\partial_j P_t+\mathrm{div}(u\partial_j
P)]\vspace{.2cm}
\\&=\mu\nabla{\dot{u}}^j\nabla[\partial_t\Delta
{u}^j+\mathrm{div}(u\Delta u^j)]
+(\lambda+\mu)\nabla\dot{u}^j\nabla[\partial_j\partial_t(\mathrm{div}u)+\mathrm{div}(u\partial_j(\mathrm{div}u))].
\end{array}
\end{equation}
Integrating the above equation over $\mathbb{R}^3$, and using
integration by parts, then we have
\begin{equation}\label{2dbu-b1.27}
\begin{array}[b]{ll}
\D\frac{1}{2}\int m|\nabla\dot{u}|^2dx &=\D \frac{1}{2}\int
m_0|\nabla\dot{u}_0|^2dx -\int_0^t\int\nabla
m\partial_t\dot{u}\nabla\dot{u}\,dxds-\int_0^t\int\nabla(mu_j)\partial_j\dot{u}\nabla\dot{u}\,dxds\vspace{.2cm}
\\\D&\ \ \
 -\int_0^t\int\nabla\dot{u}^j\nabla[\partial_j \partial_tP+\mathrm{div}(u\partial_j P)]\,dxds-\mu\int_0^t\int\Delta \dot{u}^j[\Delta u^j_t+\mathrm{div}(u\Delta u^j)]\,dxds\vspace{.2cm} \\
&\ \ \ \D-(\lambda+\mu)\int_0^t\int\Delta \dot{u}^j[\partial_j\partial_t(\mathrm{div}u)+\mathrm{div}(u\partial_j(\mathrm{div}u))]\}\,dxds\vspace{.2cm} \\
&:=\D\sum\limits_{i=1}^{6}I_i.
\end{array}
\end{equation}
From $m_0-\tilde{m}\in H^2$, and $u_0\in H^3$, we know
\begin{equation}
I_1=\frac{1}{2}\int m_0|\nabla\dot{u}_0|^2dx\leq K.
\end{equation}
From the integration by parts, the equation $(1.1)_3$,
(\ref{2dbu-b1.23}), (\ref{2dbu-b1.79}), (\ref{2dbu-b1.19}),
(\ref{2dbu-b6.7}), the H\"older inequality, the Cauchy inequality
and  the Gagliardo-Nirenberg inequality, we get
\begin{equation}
\begin{array}[b]{ll}\label{2dbu-b1.28}
I_2&=-\D\int_0^t \int\nabla m\partial_t\dot{u}\nabla\dot{u}dxds\vspace{.2cm} \\
&=-\D\int_0^t \int\nabla m\partial_t\left(\frac{\mu\Delta u+(\mu+\lambda)\nabla(\mathrm{div}u)-\nabla P}{m}\right)\nabla \dot{u}dxds\vspace{.2cm} \\
&\leq\D K\int_0^t\int|\nabla m||\nabla \dot{u}|\Big(|\nabla^2 u||\nabla u|+|\nabla u||\nabla m|+|\nabla u||\nabla n|+|\nabla m||\nabla n|+|\nabla^2 u||\nabla m|\vspace{.2cm} \\
&\ \ \ \D+|\nabla m|^2+|\nabla n|^2+|\nabla^3 u|+|\Delta \dot{u}|+|\nabla\frac{D}{Dt}\mathrm{div}u|\Big)dxds\vspace{.2cm} \\
&\leq\D K\int_0^t\{\|\nabla \dot{u}\|_{L^2}^{\frac{1}{2}}\|\nabla^2 \dot{u}\|_{L^2}^{\frac{1}{2}}(\|\nabla^2\dot{u}\|_{L^2}+\|\nabla\frac{D}{Dt}\mathrm{div}u\|_{L^2})\|m-\tilde{m} \|_{H^2}\vspace{.2cm} \\
&\ \ \ \D+\|\nabla^2 \dot{u}\|_{L^2}[(\|u\|_{H^2}^2+\|n\|_{H^2}^2+1)(\|\nabla {m}\|_{H^1}^3+1)+\|\nabla {m}\|_{H^1}\|u\|_{H^3}]\}ds\vspace{.2cm} \\
&\leq\D K+\frac{\mu}{10}\int_0^t(\|\nabla^2
\dot{u}\|_{L^2}^2+\|\nabla\frac{D}{Dt}\mathrm{div}u\|_{L^2}^2)ds.
\end{array}
\end{equation}
From the H\"older inequality, Gagliardo-Nirenberg inequality and the
Cauchy inequality, we obtain
\begin{equation}
\begin{array}[b]{ll}\label{2dbu-b1.29}
I_3&=\D-\int_0^t \int\nabla(mu^j)\partial_j \dot{u} \nabla\dot{u}\, dxds\vspace{.2cm} \\
&=\D-\int_0^t\int \nabla mu^j\partial_j\dot{u}\nabla\dot{u}+\nabla
u^j m
\partial_j\dot{u}\nabla\dot{u}\,dxds\vspace{.2cm} \\
&\leq\D K\int_0^t\int\|\nabla \dot{u}\|_{L^4}^2(\|\nabla {m}\|_{L^2}\|u\|_{L^\infty}+\|\nabla u\|_{L^2})dxds\vspace{.2cm} \\
&\leq\D    K\int_0^t\int\|\nabla \dot{u}\|_{L^2}^\frac{1}{2}\|\nabla^2 \dot{u}\|_{L^2}^\frac{3}{2}dxds\vspace{.2cm}\\
&\leq\D K+\frac{\mu}{10}\int_0^t\|\nabla^2 \dot{u}\|_{L^2}^2ds.
\end{array}
\end{equation}
From the integration by parts, the equation $(1.1)_1$, $(1.1)_2$,
(\ref{2dbu-b1.23}), (\ref{2dbu-b1.79}), (\ref{2dbu-b1.19}),
 the Cauchy inequality, we get
\begin{equation}
\begin{array}[b]{ll}\label{2dbu-b1.30}
I_4&=\D\int_0^t\int\Delta \dot{u}^j[\partial_j P_t+\mathrm{div}(\partial_j Pu)]dxds\vspace{.2cm} \\
&=\D-\int_0^t\int\left(\partial_j \Delta\dot{u}^j(P_mm_t+P_nn_t)+\partial_k \Delta\dot{u}^j\partial_j Pu^k\right)dxds\vspace{.2cm} \\
&=\D-\int_0^t\int\left([P_mm+P_nn]\mathrm{div}u\partial_j \Delta\dot{u}^j-\partial_k(\partial_j\Delta\dot{u}^j u^k)P+P\partial_j(\partial_k\Delta\dot{u}^j u^k)\right)dxds\vspace{.2cm} \\
&\leq\D K\left(\int_0^t\int(|\nabla^2u|+|\nabla m||\nabla u|+|\nabla n||\nabla u|)^2dxds\right)^\frac{1}{2}\left(\int_0^t\int|\nabla^2\dot{u}|^2dxds\right)^\frac{1}{2}\vspace{.2cm} \\
&\leq\D K+\frac{\mu}{10}\int_0^t\|\nabla^2 \dot{u}\|_{L^2}^2ds.
\end{array}
\end{equation}
From the integration by parts, the equation (\ref{2dbu-b1.23}),
(\ref{2dbu-b1.37}), (\ref{2dbu-b1.79}), (\ref{2dbu-b1.19}),
 the Cauchy inequality, we get
\begin{equation}
\begin{array}[b]{ll}\label{2dbu-b1.31}
I_5&=\D-\int_0^t\int\mu \Delta\dot{u}^j\left( \Delta{u}_t^j+\mathrm{div}(u\Delta u^j)\right)dxds\vspace{.2cm} \\
&=\D\int_0^t\int\mu\left(\partial_i\Delta\dot{u}^j\partial_i{u}_t^j+\Delta{u}^ju\cdot\nabla\Delta\dot{u}^j\right)dxds\vspace{.2cm} \\
&=\D\int_0^t\int\mu\left(\partial_i\Delta\dot{u}^j(\partial_i\dot{u}^j-\partial_i(u\cdot \nabla u^j))+\Delta{u}^ju\cdot\nabla\Delta\dot{u}^j\right)dxds\vspace{.2cm} \\
&=\D\int_0^t\int\mu\left(-|\nabla^2\dot{u}|^2-\partial_i\Delta\dot{u}^ju^k\partial_k\partial_i{u}_t^j
-\partial_i\Delta\dot{u}^j\partial_iu^k\partial_k{u}^j+\Delta{u}^ju\cdot\nabla\Delta\dot{u}^j\right)dxds\vspace{.2cm} \\
&=\D\int_0^t\int\mu\left(-|\nabla^2\dot{u}|^2+\partial_i\Delta\dot{u}^j
\mathrm{div} u\partial_i{u}^j
-\partial_i\Delta\dot{u}^j\partial_iu^k\partial_k{u}^j-\partial_i{u}^j\partial_iu^k\partial_k\Delta\dot{u}^j\right)dxds\vspace{.2cm} \\
&\leq\D -\frac{1}{2}\int_0^t\int\mu|\nabla^2\dot{u}|^2dxds+K\int_0^t\int|\nabla u|^4dxds\vspace{.2cm} \\
&\leq\D -\frac{1}{2}\int_0^t\int\mu|\nabla^2\dot{u}|^2dxds+K.
\end{array}
\end{equation}
From the integration by parts and the Cauchy inequality, we get
\begin{equation*}
\begin{array}[b]{ll}
I_6&=\D-(\lambda+\mu)\int_0^t\int\Delta{\dot{u}}^j\left(\partial_j\partial_t(\mathrm{div}u)
+\mathrm{div}(u\partial_j(\mathrm{div}u))\right)dxds\vspace{.2cm} \\
&=(\lambda+\mu)\D\int_0^t\int\left(\partial_j\Delta\dot{u}^j[\partial_t(\mathrm{div}u)+\mathrm{div}(u\mathrm{div}u)]+
\Delta\dot{u}^j\mathrm{div}(\partial_ju\mathrm{div}u)\right)dxds\vspace{.2cm} \\
&=(\lambda+\mu)\D\int_0^t\int\left(\partial_j\Delta\dot{u}^j[\partial_t(\mathrm{div}u)+\partial_ku^k\mathrm{div}u
+u^k\partial_k\mathrm{div}u)]-
\partial_i(\Delta\dot{u}^j)\partial_ju^i\mathrm{div}u\right)dxds\vspace{.2cm}\\
&\leq(\lambda+\mu)\D\int_0^t\int\partial_j\Delta\dot{u}^j\frac{D}{Dt}\mathrm{div}udxds+K\int_0^t\int|\nabla^2\dot{u}||\nabla u||\nabla^2u|dxds\vspace{.2cm} \\
&\le(\lambda+\mu)\D\int_0^t\int\left(\partial_j\Delta(\partial_tu^j+u\cdot\nabla{u}^j)\frac{D}{Dt}\mathrm{div}u\right)dxds
+K\int_0^t\|\nabla u\|_{L^\infty}\|\nabla^2\dot{u}\|_{L^2}\|\nabla^2u\|_{L^2}ds\vspace{.2cm} \\
\end{array}
\end{equation*}
\begin{equation}
\begin{array}[b]{ll}\label{2dbu-b1.32}
&\leq\D-(\lambda+\mu)\int_0^t\int|\nabla\frac{D}{Dt}\mathrm{div}u|^2dxds
+K\int_0^t\|\nabla u\|_{L^\infty}\|\nabla^2u\|_{L^2}\left(\|\nabla^2\dot{u}\|_{L^2}+\|\nabla\frac{D}{Dt}\mathrm{div}u\|_{L^2}\right)ds\vspace{.2cm} \\
&\leq\D-\frac{\lambda+\mu}{2}\int_0^t\int|\nabla\frac{D}{Dt}\mathrm{div}u|^2dxds+\frac{\mu}{10}\int_0^T\|\nabla^2\dot{u}\|_{L^2}^2ds+K.
\end{array}
\end{equation}
From (\ref{2dbu-b1.27})-(\ref{2dbu-b1.32}) and (\ref{2dbu-b1.23}),
we immediately obtain (\ref{2dbu-b1.26}).
\end{proof}

\begin{lem}\label{2dbu-a1.15}
Under the conditions of Theorem 1.1, it holds that
\begin{equation}\label{2dbu-b1.33}
\sup\limits_{t\in[0,T]}\|u\|_{H^3} \leq K.
\end{equation}
\begin{equation}\label{2dbu-b1.333}
\sup\limits_{t\in[0,T]}\|(m-\tilde{m},n-\tilde{n})\|_{H^3} \leq K.
\end{equation}
\begin{equation}\label{2dbu-b1.34}
\sup\limits_{t\in[0,T]}\|(m_t,n_t)\|_{H^2}\leq K.
\end{equation}
\begin{equation}\label{2dbu-b1.35}
\int_0^T\left(\|u\|_{H^4}^2+\|u_t\|_{H^2}^2\right)dt\leq K.
\end{equation}
\end{lem}
\begin{proof}
From (\ref{2dbu-w1.55}), (\ref{2dbu-c1.5}), (\ref{2dbu-c1.6}),
(\ref{2dbu-q1.1}), (\ref{2dbu-b1.19}), (\ref{2dbu-b1.26}) and the
Gagliardo-Nirenberg inequality, we have
\begin{equation}\label{2dbu-b1.41}
\begin{array}[b]{ll}
\D\| u\|_{H^3}&\leq K\D(\| u\|_{L^2}+\|\nabla F\|_{H^1}+\| \nabla (P(m,n)-P(\tilde{m},\tilde{n}))\|_{H^1}+\|\nabla \omega\|_{H^1})\vspace{.2cm} \\
&\leq K\D(\| u\|_{L^2}+\|\nabla (m\dot{u})\|_{L^2}+\| m-\tilde{m}\|_{H^2}+\| n-\tilde{n}\|_{H^2})\vspace{.2cm} \\
&\leq K\D(\| u\|_{L^2}+\|\nabla \dot{u}\|_{L^2}+\|\nabla m \dot{u}\|_{L^2}+\| m-\tilde{m}\|_{H^2}+\| n-\tilde{n}\|_{H^2})\vspace{.2cm} \\
&\leq K.
\end{array}
\end{equation}
Thus, we can get (\ref{2dbu-b1.33}).

Differentiating (\ref{2dbu-b1.20}) with respect to $x_i$ , $x_j$,
and $x_k$, multiplying both sides of the resulting equation by
$\partial_i\partial_j\partial_k\Lambda_1(m)$, integrating the result
equality over $\mathbb{R}^3$, we obtain
\begin{equation}
\begin{array}[b]{ll}
\D\frac{1}{2}\frac{d}{dt}\int
|\partial_i\partial_j\partial_k\Lambda_1(m)|^2dx&\D\leq\int
(|\partial_i\partial_j\partial_k\Lambda_1(m)||\partial_i\partial_j\partial_kF|\vspace{.2cm} \\
&\ \ \
+|\partial_i\partial_j\partial_k\Lambda_1(m)||\partial_i\partial_j\partial_kP|+|\partial_i\partial_j\partial_k\Lambda_1(m)|
|\partial_i\partial_j\partial_k(u\cdot\nabla\Lambda_1)|)dx.
\end{array}
\end{equation}
Using the Cauchy inequality and the Gagliardo-Nirenberg inequality,
we have
\begin{equation}\label{2dbu-b1.42}
\|\Lambda_1(t)\|_{H^3}^2\leq\|\Lambda_1(0)\|_{H^3}^2+K\int_0^t((\|F\|_{H^3}+\|\Lambda_2\|_{H^3})\|\Lambda_1\|_{H^3}+(
1+\|\nabla u\|_{L^\infty}+\|u\|_{H^3})\|\Lambda_1\|_{H^3}^2)ds.
\end{equation}
Similarly, from (\ref{2dbu-b1.222}), we have
\begin{equation}\label{2dbu-b1.43}
\|\Lambda_2(t)\|_{H^3}^2\leq\|\Lambda_2(0)\|_{H^3}^2+K\int_0^t((\|F\|_{H^3}+\|\Lambda_1\|_{H^3})\|\Lambda_2\|_{H^3}+(
1+\|\nabla u\|_{L^\infty}+\|u\|_{H^3})\|\Lambda_2\|_{H^3}^2)ds.
\end{equation}
From (\ref{2dbu-c1.5}), (\ref{2dbu-b1.8}), (\ref{2dbu-b1.80}),
(\ref{2dbu-b1.19}), (\ref{2dbu-b1.26}) and the Gagliardo-Nirenberg
inequality, we have
\begin{equation}\label{2dbu-b1.44}
\begin{array}[b]{ll}
\|F\|_{H^3}&\leq K(\|F\|_{L^2}+\|\nabla m \dot{u}\|_{H^1}+\| m \nabla\dot{u}\|_{H^1})\vspace{.2cm} \\
&\leq K(1+\|\nabla^2 m \dot{u}\|_{L^2}+\|\nabla^2\dot{u}\|_{L^2}+\|\nabla m \nabla\dot{u}\|_{L^2})\vspace{.2cm} \\
&\leq K(1+\|\nabla^2\dot{u}\|_{L^2}+\|\nabla^2 m \|_{L^3}\|\dot{u}\|_{L^6}+\|\nabla m \|_{L^3}\|\nabla \dot{u}\|_{L^6})\vspace{.2cm} \\
&\leq K(1+\|\nabla^2\dot{u}\|_{L^2}+\|\nabla^2 m
\|_{L^2}^{\frac{1}{2}}\|\nabla^3 m
\|_{L^2}^{\frac{1}{2}}\|\nabla\dot{u}\|_{L^2}+\|\nabla m
\|_{L^2}^{\frac{1}{2}}\|\nabla^2
m \|_{L^2}^{\frac{1}{2}}\|\nabla^2 \dot{u}\|_{L^2})\vspace{.2cm} \\
&\leq K(1+\|\nabla^2 \dot{u}\|_{L^2}+\|\Lambda_1\|_{H^3}).
\end{array}
\end{equation}
Thus, from  (\ref{2dbu-b1.26}), (\ref{2dbu-b1.33}),
 (\ref{2dbu-b1.42}), (\ref{2dbu-b1.43}),
and (\ref{2dbu-b1.44}), we have
\begin{equation}\label{2dbu-b1.45}
\begin{array}{ll}
&\ \ \ \|\Lambda_1(t)\|_{H^3}^2+\|\Lambda_2(t)\|_{H^3}^2
\vspace{.2cm}\\
&\D\leq
\|\Lambda_1(0)\|_{H^3}^2+\|\Lambda_2(0)\|_{H^3}^2+K\int_0^t((1+\|\nabla
u\|_{L^\infty}+\|\nabla^2\dot{u}\|_{L^2}+\|u\|_{H^3})(\|\Lambda_1\|_{H^3}^2+\|\Lambda_2\|_{H^3}^2)+
\|\nabla^2\dot{u}\|_{L^2}^2)ds\vspace{.2cm} \\
&\leq
K+K\D\int_0^t(1+\|\nabla^2\dot{u}\|_{L^2})(\|\Lambda_1\|_{H^3}^2+\|\Lambda_2\|_{H^3}^2)ds,
\end{array}
\end{equation}
where we have used $m_0-\tilde{m}\in H^3$, $n_0-\tilde{n}\in H^3$.

Using (\ref{2dbu-b1.26}), (\ref{2dbu-b1.45}) and the Gronwall's
inequality, we have
\begin{equation}\label{2dbu-b1.46}
\|\Lambda_1(t)\|_{H^3}^2+\|\Lambda_2(t)\|_{H^3}^2\leq K.
\end{equation}
Thus, we can immediately obtain (\ref{2dbu-b1.333}).

From $(1.1)_1$ , $(1.1)_2$ , (\ref{2dbu-b1.33}) and
(\ref{2dbu-b1.333}), we get
\begin{equation*}
\sup\limits_{t\in[0,T]}\|(m_t,n_t)\|_{H^2}\leq K.
\end{equation*}
Then we  get (\ref{2dbu-b1.34}).

From  (\ref{2dbu-c1.6}), (\ref{2dbu-b1.26}), (\ref{2dbu-b1.46}) and
(\ref{2dbu-b1.44}), we have
\begin{equation}\label{2dbu-b1.47}
\int_0^T\left(\| F\|_{H^3}^2+\| \omega\|_{H^3}^2\right)dt\leq K.
\end{equation}
From (\ref{2dbu-w1.55}), (\ref{2dbu-q1.1}), (\ref{2dbu-b1.333}) and
(\ref{2dbu-b1.47}), we have
\begin{equation}\label{2dbu-b1.25}
\begin{array}[b]{ll}
\D\int_0^T\| u\|_{H^4}^2dt&\leq K\int_0^T\left(\| u\|_{L^2}^2+\|\nabla F\|_{H^2}^2+\| \nabla (P(m,n)-P(\tilde{m},\tilde{n}))\|_{H^2}^2+\|\nabla \omega\|_{H^2}^2\right)dt\vspace{.2cm} \\
&\leq K\D\int_0^T\left(\| u\|_{L^2}^2+\| F\|_{H^3}^2+\| m-\tilde{m}\|_{H^3}^2+\| n-\tilde{n}\|_{H^3}^2+\| \omega\|_{H^3}^2\right)dt\vspace{.2cm} \\
&\leq K.
\end{array}
\end{equation}
From $\dot{u}=u_t+u\cdot\nabla u$, (\ref{2dbu-b1.26}) and
(\ref{2dbu-b1.33}), we have
\begin{equation}\label{2dbu-b1.35}
\int_0^T\|u_t\|_{H^2}^2dt\leq K.
\end{equation}
Then we get (\ref{2dbu-b1.35}).

This completes the proof of Lemma \ref{2dbu-a1.15}.
\end{proof}

{\noindent\bf Step 3:} Completion of the proof of Theorem 1.1.\\

By Lemmas \ref{2dbu-a1.0}, \ref{2dbu-a1.8} and \ref{2dbu-a1.15}, we
get (\ref{aim1}) and (\ref{aim2}), which concludes a contradiction.
Thus, $T^*=\infty$. The proof of Theorem 1.1 is complete.
\section*{Acknowledgements} The authors would like to thank the anonymous referees for their constructive suggestions and
kindly comments. The authors also thank professor Changjiang Zhu and
Dr. Lei Yao for their helpful discussion. This work was supported by
the National Natural Science Foundation of China $\#$10625105,
$\#$11071093, the PhD specialized grant of the Ministry of Education
of China $\#$20100144110001, and the Special Fund for Basic
Scientific Research  of Central Colleges $\#$CCNU10C01001.


\begin{thebibliography}{99}

\bibitem{C}
Y. Cho, H. Kim, On classical solutions of the compressible
Navier-Stokes equations with nonnegative initial densities,
Manuscripta math., 120(2006), 91-129.

\bibitem{EV0}
S. Evje, T. Fl{\aa}tten, H.A. Friis, Global weak solutions for a
viscous liquid-gas model with transition to single-phase gas flow
and vacuum, Nonlinear Anal., TMA, 70(2009), 3864-3886.

\bibitem{EV1}
S. Evje, K.H. Karlsen, Global existence of weak solutions for a
viscous two-phase model, J. Differential Equations, 245(2008), 2660-2703.


\bibitem{EV2} S. Evje, K.H. Karlsen, Global weak solutions for a
viscous liquid-gas model with singular pressure law, Commun. Pure
Appl. Anal., 8(2009), 1867-1894.


\bibitem{FS} H.A. Friis, S. Evje, T. Flatten,
A numerical study of characteristic slow-transicent behavior of a
compressible 2D gas-liquid two-fluid model, Adv. Appl. Math. Mech.,
1(2009), 166-200.

\bibitem{Guo-Yang-Yao}Z.H. Guo, J. Yang, L. Yao, Global strong solution for a three-dimensional viscous
liquid-gas two-phase flow model with vacuum, J. Math. Phy., 52,
093102(2011).

\bibitem{Hoff2005}
D. Hoff, Compressible flow in a half-space with Navier boundary conditons,
J. Math. Fluid Mech, 7(1995), 315-338.

\bibitem{Hoff1995}
D. Hoff, Global solutions of the Navier-Stokes equations for
multidimensional compressible flow with discontinuous initial data,
J. Differential Equations, 120(1995), 215-254.

\bibitem{Sun-Wang-Zhang}
Y.Z. Sun, C. Wang, Z.F. Zhang, A Beale-Kato-Majda blow-up criterion
for the 3-D compressible Navier-Stokes equations, J. Math. Pures
Appl., 95(2011), 36-47.

\bibitem{Sun-Zhang}
Y.Z. Sun, Z.F. Zhang, A blow-up criterion of strong solution for the
2-D compressible Navier-Stokes equations, Sci. China Math., 54
(2011), 105-116.

\bibitem{W-Zhu2}
H.Y. Wen, L. Yao, C.J. Zhu, {A blow-up criterion of strong solution
to a 3D viscous liquid-gas two-phase flow model with vacuum}, J.
Math. Pures Appl., 97(2012), 204-229.

\bibitem{Yao-Zhu}
L. Yao, C.J. Zhu, {Free boundary value problem for a viscous
two-phase model with mass-dependent viscosity}, J. Differential
Equations, 247(2009), 2705-2739.

\bibitem{Yao-Zhu2}
L. Yao, C.J. Zhu, {Existence and uniqueness of global weak solution
to a two-phase flow model with vacuum}, Math. Ann., 349(2011),
903-928.

\bibitem{Y.Z.Z}
L. Yao, T. Zhang, C.J. Zhu,  Existence and asymptotic behavior of
global weak solutions to a  2D viscous liquid-gas two-phase flow
model, SIAM J. Math. Anal., 42(2010), 1874-1897.

\bibitem{Z.F}
T. Zhang, D.Y. Fang,  Compressible flows with a density-dependent
viscosity coefficient, SIAM J. Math. Anal., 41(2009), 2453-2488.

\bibitem{Z}
T. Zhang,  Global solutions of compressible barotropic Navier-Stokes
equations with a density-dependent viscosity coefficient, J. Mathematical Physics, 52, 043510 (2011), 1-26.


\end{thebibliography}
\end{document}